\documentclass[12pt, leqno]{amsart}

\usepackage{graphicx}
\usepackage{amsfonts,delarray,amssymb,amsmath,amsthm,a4,a4wide}
\usepackage{latexsym}
\usepackage{epsfig}
\usepackage{color}
\usepackage[margin=1in]{geometry} 

\newtheorem{thm}{Theorem}[section]

\newtheorem{lem}[thm]{Lemma}
\newtheorem{prop}[thm]{Proposition}
\newtheorem{rem}[thm]{Remark}

\theoremstyle{definition}

\numberwithin{equation}{section}

\newcommand{\R}{\mathbb R}

\newcommand{\e}{\varepsilon}

\newcommand{\p}{\partial}
\newcommand{\dist}{\mbox{dist}\,}
\newcommand{\diam}{\mbox{diam}\,}
\newcommand{\trace}{\mbox{trace}\,}

\newcommand{\comment}[1]{}
\def\h{\hspace*{.24in}}

\usepackage{esvect}
\newenvironment{myindentpar}[1]%
{\begin{list}{}%
         {\setlength{\leftmargin}{#1}}%
         \item[]%
}
{\end{list}}

\makeatletter
\@namedef{subjclassname@2020}{%
  \textup{2020} Mathematics Subject Classification}
\makeatother
\begin{document} 

\title[Singular Monge-Amp\`ere equations on general  bounded convex domains]{Optimal boundary regularity for some singular  Monge-Amp\`ere equations on bounded convex domains}
\author{Nam Q. Le}
\address{Department of Mathematics, Indiana University, 831 E 3rd St,
Bloomington, IN 47405, USA}
\email{nqle@indiana.edu}
\thanks{The author was supported in part by the National Science Foundation under grant DMS-1764248.}

\subjclass[2020]{ 35J96, 35J75}
\keywords{Singular Monge-Amp\`ere equation, optimal global H\"older regularity, supersolution, complete affine hyperbolic spheres, proper affine hyperspheres}

\maketitle
\begin{abstract}
 By constructing explicit supersolutions, we obtain the optimal global H\"older regularity 
 for several singular Monge-Amp\`ere equations on general bounded open convex domains including those related to complete affine hyperbolic spheres, and proper affine hyperspheres. Our analysis reveals that certain singular-looking equations, such as $ \det D^2 u = |u|^{-n-2-k} (x\cdot Du -u)^{-k} $ with zero boundary data, have unexpected degenerate nature.
 \end{abstract}

\section{Introduction and  statements of the main results}

This note is concerned with  the optimal global H\"older regularity  of the unique convex solution $u\in C^{\infty}(\Omega)\cap C(\overline{\Omega})$ to some singular Monge-Amp\`ere equations whose prototype is
\begin{equation}
 \label{SMA}
 \left\{
 \begin{alignedat}{2}
   \det D^{2} u~&= |u|^{-p} \h~&&\text{in} ~\Omega, \\\
u &=0\h~&&\text{on}~\p \Omega
 \end{alignedat}
 \right.
\end{equation}
where $p>0$ and $\Omega$ is a general bounded open convex domain in $\R^n$ ($n\geq 2$). Using suitable subsolutions, we show that $u\in C^{\frac{2}{n+ p}}(\overline{\Omega})$. By constructing explicit supersolutions on some special bounded convex domains, we show that this global H\"older regularity is optimal when $p\geq 1$. The class of such bounded convex domains includes those containing parts of hyperplanes on their boundaries.

There is an extensive literature on (\ref{SMA}) when $\Omega$ is a $C^2$, bounded, strictly convex domain. Our focus here is on the general bounded convex domains which can contain parts of hypersurfaces on their boundaries. To the best of the author's knowledge, the only existence result available for (\ref{SMA}) is the work of Cheng and Yau \cite{CY1} for $p=n+2$. In this case, if $u$ is a solution to (\ref{SMA}), then the Legendre transform of $u$ is a complete affine hyperbolic sphere; see \cite{Cal, CY1, CY2}. 
Also related to the exponent $p=n+2$ is the work of Chen-Huang \cite{CH} who treated the following equation for all $k\geq 0$
 \begin{equation}
 \label{SMAnk2}
\left\{
 \begin{alignedat}{2}
   \det D^2 u~& = |u|^{-n-2-k} (x\cdot Du -u)^{-k} ~&&\text{in} ~  \Omega, \\\
u &= 0~&&\text{on}~ \p\Omega,
 \end{alignedat} 
  \right.
  \end{equation}
  which is related to proper affine hyperspheres. When $k>0$, the domain $\Omega$ in (\ref{SMAnk2}) was required to contain the origin so that the expression $x\cdot Du-u$ is positive.

  Other positive values of $p$ in (\ref{SMA}) are also of interest because of the close relation with the 
$L_p$-Minkowski problem (see,  for example \cite{Luk}) and the Minkowski problem in centro-affine geometry (see, for example \cite{CW, JW} and the references therein).
The case $p=1$ is closely related to the logarithmic Minkowski problem \cite{BLYZ}.
For $p>0$, 
the existence of a unique convex solution to (\ref{SMA}) is perhaps well known to experts. However, since we could not locate a reference and for the reader's convenience, we include its existence proof in part (i) of Theorem \ref{SMApthm} below.
\vglue 0.1cm

Our first main result states as follows.
\begin{thm}
  \label{SMApthm}
  Let $\Omega$ be a bounded, open convex domain in $\R^n$ ($n\geq 2$). 
  Let $p>0$. 
  Then,
  \begin{myindentpar}{1cm}
  (i) there exists a unique convex solution $u\in C^{\infty}(\Omega)\cap C(\overline{\Omega})$ to
  \begin{equation}
  \label{SMAp}
\left\{
 \begin{alignedat}{2}
   \det D^2 u~& = |u|^{-p}~&&\text{in} ~  \Omega, \\\
u &= 0~&&\text{on}~ \p\Omega.
 \end{alignedat} 
  \right.
  \end{equation}
 (ii) 
  moreover, $u\in C^{\frac{2}{n+p}}(\overline{\Omega})$ and the following estimate holds:
  $$|u(x)| \leq C(n, p,\text{diam} (\Omega)) [\dist(x,\p\Omega)]^{\frac{2}{n+ p}} \quad\text{for all } x\in\Omega.$$
  (iii) if $p\geq 1$, then the global H\"older regularity in (ii) is optimal in the following sense: there exist bounded convex domains $\Omega\subset\R^n$ such that $u\not\in C^\beta(\overline{\Omega})$ for any $\beta>\frac{2}{n+p}$. In fact, the class of such bounded convex domains includes those containing parts of hyperplanes on their boundaries.
    \end{myindentpar}
\end{thm}
 
Note that the global H\"older estimate in the  case of $p=n+2$ 
in Theorem \ref{SMApthm} (ii) was treated by Lin-Wang \cite[Section 4]{LW} and also by Jian-Li \cite[Theorem 1.2]{JL_JDE18}. Jian-Li observed that when $n= 2$ and $\Omega=(-1, 1)\times \R_{+}\subset\R^2$, the $C^{\infty}(\Omega)\cap C^{\frac{1}{3}}(\overline{\Omega})$ convex function 
\begin{equation}
\label{uJL}
u(x_1, x_2)= -\sqrt{3} \left(\frac{x_2}{2}\right)^{\frac{1}{3}} (1-x_1^2)^{\frac{1}{3}}
\end{equation} is a solution of $\det D^2 u= |u|^{-4}$ in $\Omega$ with $u=0$ on $\p\Omega$. This example, though for unbounded $\Omega$,
indicates that  the global estimates in Theorem \ref{SMApthm} (ii) should be optimal, at least when $p=n+2$. The bounded convex domains $\Omega$ and the supersolutions constructed in Lemma \ref{suplem} actually confirm the optimality of the global H\"older estimates in Theorem \ref{SMApthm} (ii) for all $p\geq 1$. Our ansatz is quite different from the ansatz in (\ref{uJL}); see also Remarks \ref{p1rem} and \ref{suprem2}. It is an open question to investigate the remaining case $0<p<1$.
\begin{rem}
Since our method of proof of optimality relies only on explicit supersolutions but not on explicit solutions, it is quite flexible and can be extended to (\ref{SMAnk2}); see Theorem \ref{SMAkthm}. When $k>0$, we are not aware of any explicit solution to (\ref{SMAnk2}) on any convex domain $\Omega$. During the course of the proof of Theorem \ref{SMApthm} (iii), we find that when $p=n+ 2$ or $p=1$, (\ref{SMA}) has an explicit solution when $\Omega$ is a circular cylinder; see Remarks \ref{p1rem} and \ref{suprem2}.
\end{rem}

The proof of Theorem \ref{SMApthm} (ii) is based on a subsolution as in \cite{LSNS} which treated (\ref{SMA}) in the case $p= -q<0$; see also \cite{CH} for similar arguments. Interestingly, this argument gives almost the same H\"older exponent for $p>0$ and $p<0$.

When $p=-q<0$, we have the following degenerate Monge-Amp\`ere equation
 \begin{equation}
  \label{DMAq}
\left\{
 \begin{alignedat}{2}
   \det D^2 u~& = |u|^q~&&\text{in} ~  \Omega, \\\
u &= 0~&&\text{on}~ \p\Omega.
 \end{alignedat} 
  \right.
  \end{equation}
 For $0<q\neq n$, the existence of a nonzero convex solution to (\ref{DMAq}) was established by Tso \cite[Corollary 4.2 and Theorem E]{Tso} when $\Omega$ is a bounded, smooth and uniformly convex domain, and by the author \cite[Theorem 4.2]{LSNS} when $\Omega$ is a general bounded convex domain. When $0\leq q<n$, the nonzero solution $u$ is unique. The case $q=n$ corresponds to the Monge-Amp\`ere eigenvalue problem \cite{Ls, LSNS}.   By \cite[Proposition 2.8]{LSNS}, if $u\in C(\overline{\Omega})$ is an Aleksandrov solution to (\ref{DMAq}), then $u\in C^{\infty}(\Omega)$.
 
 If $\Omega$ is an open bounded convex domain in $\R^n$ and $n-2\leq q\neq n$, the author established in \cite[Proposition 5.4]{LSNS} the global almost Lipschitz property for solutions to (\ref{DMAq}), that is, $u\in C^{0,\beta}(\overline{\Omega})$
 for all $\beta\in (0, 1)$.
It is not known if this  global almost Lipschitz result holds for the nonzero solution to (\ref{DMAq}) when $n\geq 3$, $q\in (0, n-2)$ and $\Omega$ is a 
general open bounded convex domain in $\R^n$. The following shows that, in the remaining cases, we have the global almost $C^{\frac{2}{n-q}}$ property.

\begin{prop}
\label{q0n2}
Let $\Omega$ be an open bounded convex domain in $\R^n$ where $n\geq 3$ and $0<q<n-2$.
 Let $u\in C^{\infty}(\Omega)\cap C(\overline{\Omega})$ be the nonzero convex solution of 
  \begin{equation*}
\left\{
 \begin{alignedat}{2}
   \det D^2 u~& = |u|^q~&&\text{in} ~  \Omega, \\\
u &= 0~&&\text{on}~ \p\Omega.
 \end{alignedat} 
  \right.
  \end{equation*}
Then, 
 for all $\beta\in (0, \frac{2}{n-q})$, we have $u\in C^{0,\beta}(\overline{\Omega})$ and the estimate
 \begin{equation}
 \label{uqb}
 |u(x)|\leq C(n, q, \beta, \emph{diam} \Omega) [\emph{dist}(x,\p\Omega) ]^{\beta}~\text{for all}~x\in\Omega.
 \end{equation}
\end{prop}
When $k>0$ and $\Omega$ is a bounded, open convex domain in $\R^n$ $(n\geq 2)$ that contains the origin in its interior, Chen and Huang \cite{CH} proved the existence of a  unique convex solution $u\in C^{\infty}(\Omega) \cap  C(\overline{\Omega})$ to (\ref{SMAnk2}). Their proof (see \cite[p. 871]{CH}) shows that $u\in C^{\frac{1}{n+ k+1}}(\overline{\Omega})$. Exploring the fact that $\dist(0,\p\Omega)>0$ together with suitable constructions of subsolutions and supersolutions, we show a higher global H\"older regularity $u\in C^{\frac{2+k}{2n+ 2k+2}}(\overline{\Omega})$ and this is in fact optimal. 

Our final theorem states as follows.
\begin{thm} 
\label{SMAkthm}
Let $\Omega$ be a bounded, open convex domain in $\R^n$ $(n\geq 2)$ that contains the origin in its interior. Let $k> 0$. Let $u\in C^{\infty}(\Omega)\cap C(\overline{\Omega})$ be the unique convex solution to 
\begin{equation}
 \label{SMAnk22}
\left\{
 \begin{alignedat}{2}
   \det D^2 u~& = |u|^{-n-2-k} (x\cdot Du -u)^{-k} ~&&\text{in} ~  \Omega, \\\
u &= 0~&&\text{on}~ \p\Omega.
 \end{alignedat} 
  \right.
  \end{equation}
 Then
$u\in C^{\frac{2+k}{2n+ 2k+2}}(\overline{\Omega})$ with the estimate
\begin{equation}
\label{uHk}
|u(x)|\leq C(n, k, \diam (\Omega), \dist(0,\p\Omega) [\dist(x,\p\Omega)]^{\frac{2+k}{2n+ 2k+2}} ~\text{for all}~x\in\Omega,
\end{equation}
and the exponent $\frac{2+k}{2n+ 2k+2}$ is optimal.
\end{thm}

\begin{rem} We now compare the results in Theorems \ref{SMApthm} and \ref{SMAkthm}.
\begin{myindentpar}{1cm}
(i) If the right hand-side $|u|^{-p}$ of (\ref{SMAp}) is replaced by $f|u|^{-p}$ where $f$ is a positive bounded function on $\overline{\Omega}$ then, from the proof of Theorem \ref{SMApthm} (ii), the global $C^{\frac{2}{n+ p}}(\overline{\Omega})$ estimate for $u$ still holds. \\
(ii) In the context of Theorem \ref{SMAkthm} where one can show that $x\cdot Du-u$ is bounded from below by a positive constant, if we just use the boundedness of $(x\cdot Du-u)^{-k}$, then the right hand-side of (\ref{SMAnk22}) is bounded from above by $C|u|^{-n-k-2}$ from which we can use (i) to deduce $u\in C^{\frac{1}{n+ k+1}}(\overline{\Omega})$. This result is due to Chen and Huang as mentioned earlier.\\
(iii) Our higher global H\"older regularity for $u$ in Theorem \ref{SMAkthm} is based on the insight that $Du$ is unbounded near the boundary $\p\Omega$ so $(x\cdot Du-u)^{-k}$ tends to $0$ near the boundary with certain rate. The estimate (\ref{vaDva}) in the proof of Lemma \ref{vallemk} suggests that the rate is at least $[\dist(\cdot, \p\Omega)]^{\frac{(2n+k)k}{2n+ 2k+ 2}}.$
\end{myindentpar}
\end{rem}
Surprisingly, despite the singular-looking nature of the right-hand side of  (\ref{SMAnk22}), we find that when $n>4$ and $k$ is large, right-hand side of  (\ref{SMAnk22}) is not completely singular  as it might tend to $0$ near the boundary. Thus (\ref{SMAnk22}) can have a degeneracy nature as in (\ref{DMAq}).
The following proposition makes this more precise.
\begin{prop}
\label{mixc}
Assume $k> 0$ and $0<\gamma<1$. 
Let $$\Omega=\{(x', x_n): |x^{'}|<1, 0<x_n+\gamma < 1- |x^{'}|^2\}.$$
Let $u\in C^{\infty}(\Omega)\cap C(\overline{\Omega})$ be the unique convex solution to (\ref{SMAnk22}). 
Let $$f(x):= |u(x)|^{-n-2-k} (x\cdot Du(x) -u(x))^{-k}.$$
Then, for $x= (0, x_n)\in\Omega$ with $0<x_n+\gamma$ small, we have
\begin{equation}
\label{fmix}
f(x) \leq  C(n, k,\gamma) [\dist(x,\p\Omega]^{\frac{(n-4) k -(2n+ 4)}{2n+ 2k+ 2}}.
\end{equation}
In particular, if $n>4$ and $k>\frac{2n+ 4}{n-4}$, then the right-hand side of (\ref{SMAnk22}) tends to $0$ as $x$ approaches the boundary $\p\Omega\cap\{x_n=-\gamma\}$ along the $x_n$-axis.
\end{prop}
In this note, we denote a point $x= (x_1, \cdots, x_n)\in\R^n$
 by $(x', x_n)$ where $x'=(x_1, \cdots, x_{n-1})$. In computations, we usually denote $r=|x'|$. The Lebesgue measure of $\Omega$ is denoted by $|\Omega|$.

The rest of this note is organized as follows: In Section \ref{subsup}, we construct subsolutions and supersolutions of singular Monge-Amp\`ere equations (\ref{SMAp}). In Section \ref{pfsect}, we prove Theorem \ref{SMApthm} and Proposition \ref{q0n2}.
The proofs of Theorem \ref{SMAkthm} and Proposition \ref{mixc} will be given in Section \ref{ksect}.
  
\section{Subsolutions and supersolutions of singular Monge-Amp\`ere equations}
\label{subsup}
In the proof of Theorem \ref{SMApthm}, we frequently use the following comparison principle whose short proof is included for completeness.
\begin{lem}[Comparison principle]
\label{complem} Let $\Omega\subset\R^n$ be a bounded convex domain.
 Let $p>0$. 
 \begin{myindentpar}{1cm}
 (i) Assume that $u, v\in C^2(\Omega)\cap C(\overline{\Omega})$ are convex functions with
 $$\det D^2 u\geq |u|^{-p},\quad \det D^2 v\leq |v|^{-p}\quad\text{in }\Omega$$
 and $0\geq v\geq u$ on $\p\Omega$. Then $v\geq u$ in $\Omega$.\\
 (ii) Let $k>0$. Assume that $u, v\in C^2(\Omega)\cap C(\overline{\Omega})$ are convex functions with $$x\cdot Du-u>0, \quad x\cdot Dv-v>0 \quad \text{in }\Omega,\quad 0\geq v\geq u \quad\text{on }\p\Omega$$ and
 $$\det D^2 u\geq |u|^{-p} (x\cdot Du-u)^{-k},\quad \det D^2 v\leq |v|^{-p}(x\cdot Dv-v)^{-k}\quad\text{in }\Omega.$$
Then $v\geq u$ in $\Omega$.
 \end{myindentpar}
\end{lem}
\begin{proof} (i) If $v-u$ attains its minimum value on $\overline{\Omega}$ at $x_0\in\Omega$ with $v(x_0)<u(x_0)<0$, then $D^2 v(x_0)\geq D^2 u(x_0)$. It follows that
$$|v(x_0)|^{-p}\geq \det D^2 v(x_0)\geq \det D^2 u(x_0) \geq |u(x_0)|^{-p}.$$
Therefore, $|v(x_0)|^{-p}\geq |u(x_0)|^{-p}$ which contradicts $|v(x_0)|>|u(x_0)|$ and $p>0$.\\
(ii) If $v-u$ attains its minimum value on $\overline{\Omega}$ at $x_0\in\Omega$ with $v(x_0)<u(x_0)<0$, then $Du(x_0)= Dv(x_0)$ and $D^2 v(x_0)\geq D^2 u(x_0)$. From  $ \det D^2 v(x_0)\geq \det D^2 u(x_0)$ and the assumptions on $v$ and $u$, we deduce that
$$(x_0\cdot Dv(x_0)-v(x_0))^{-k} |v(x_0)|^{-p}\geq |u(x_0)|^{-p} (x_0\cdot Du(x_0)-u(x_0))^{-k}.$$
Since $k>0$ and
$$x_0\cdot Dv(x_0)-v(x_0)=x_0\cdot Du(x_0)-v(x_0)>x_0\cdot Du(x_0)-u(x_0),$$
we easily find $|v(x_0)|^{-p}> |u(x_0)|^{-p}$ which contradicts $|v(x_0)|>|u(x_0)|$ and $p>0$.
\end{proof}
The following lemma, motivated by \cite[Lemma 1]{C1}, constructs subsolutions to (\ref{SMAp}).
\begin{lem}[Subsolutions for (\ref{SMAp})]
\label{vallem}
Let $\Omega$ be a bounded convex domain such that $0\in \p\Omega$ and $\Omega\subset \R^n_+=\{x=(x', x_n)\in\R^n: x_n>0\}$.
For $\alpha\in (0, 1)$, we consider the following function on $\Omega$
\begin{equation}
\label{valphaeq}
v_{\alpha}(x)= x_n^{\alpha} (|x'|^2 -C_\alpha)~\text{where}~C_{\alpha}= \frac{1 +2 [\diam\Omega]^2}{\alpha (1-\alpha)}.
\end{equation}
Then
\begin{myindentpar}{1cm}
(i) $v_\alpha$ is convex in $\Omega$ with $v_\alpha\leq 0$ on $\p\Omega$ and
$$\det D^2 v_\alpha(x)\geq 2 x_n^{n\alpha-2}\quad\text{in }\Omega.$$
(ii) 
$$\det D^2 v_{\frac{2}{n+ p}}> |v_{\frac{2}{n+ p}}|^{-p} \quad\text{if } p>0.$$
\end{myindentpar}
\end{lem}
\begin{proof}
(i) Note that
$$D_{ij}v_{\alpha} = 2x_n^{\alpha}\delta_{ij}~\text{for}~i, j\leq n-1;~ D_{in}v_{\alpha}= 2\alpha x_i x_n^{\alpha-1}~ \text{and}~ D_{nn} v_{\alpha} = \alpha (\alpha-1) (|x'|^2-C_\alpha)x_n^{\alpha-2},$$
where $\delta_{ij}$ is the Kronecker symbol, that is,  $\delta_{ij}=1$ if $i=j$ and $\delta_{ij}=0$ if $i\neq j$.

A short computation (see also the computation in the proof of Lemma \ref{suplem}) gives
\begin{equation}
\label{detval}
\det D^2 v_{\alpha} (x)= 2^{n-1}x_n^{n\alpha-2} [\alpha(1-\alpha) C_\alpha- (\alpha^2+ \alpha) |x'|^2]\quad \text{for }x\in\Omega.
\end{equation}
Therefore (i) easily follows, since, from the definition of $C_\alpha$, we have 
$$\alpha(1-\alpha) C_\alpha- (\alpha^2+ \alpha) |x'|^2 \geq 1. $$
(ii) Now consider $p>0$ and $$\alpha:= \frac{2}{n+ p}.$$
Using $\alpha (1-\alpha)\leq \frac{1}{4}$, we deduce that
$C_{\alpha} \geq  4 + 8  [\diam\Omega]^2.$
Therefore, in view of (\ref{detval}) and $n\geq 2$, we find
\begin{eqnarray*}
|v_\alpha|^p\det D^2 v_\alpha &>& |v_\alpha|^p x_n^{n\alpha-2} [\alpha(1-\alpha) C_\alpha- (\alpha^2+ \alpha) |x'|^2] \\&=&  [\alpha(1-\alpha) C_\alpha- (\alpha^2+ \alpha) |x'|^2](C_\alpha - |x'|^2)^p\\
&\geq&  (C_\alpha - |x'|^2)^p>1.
\end{eqnarray*}
\end{proof}
To establish a lower bound for the $L^{\infty}$ norm of solution to (\ref{SMAp}), we use the following lemma.
\begin{lem}
\label{ulb0}
Let $\Omega\subset\R^n$ be a bounded convex domain in $\R^n$. Let $p,\e\geq 0.$
Let $u\in C^{\infty}(\Omega)\cap C(\overline{\Omega})$ be the convex solution to
 \begin{equation*}
\left\{
 \begin{alignedat}{2}
   \det D^2 u~& = (|u|+\e)^{-p}~&&\text{in} ~  \Omega, \\\
u &= 0~&&\text{on}~ \p\Omega.
 \end{alignedat} 
  \right.
  \end{equation*}
  If $\e<\e_0(n, p, |\Omega|)$, then
  $$\|u\|_{L^{\infty}(\Omega)}\geq c(n, p)|\Omega|^{\frac{2}{n+p}}.$$
  \end{lem}
  \begin{proof} The proof is similar to that of \cite[Lemma 3.1 (iii)]{LSNS}.
  Under the affine transformation $T:\R^n\rightarrow\R^n$ with $\det T=1$:
 $$\Omega\rightarrow T(\Omega),~u(x)\rightarrow u(T^{-1}x),$$
 the equation $ \det D^2 u = (|u|+\e)^{-p}$, the quantities 
 $\|u\|_{L^{\infty}(\Omega)}~ \text{and}~ |\Omega|$
 are unchanged. Thus, by John's lemma,
 we can assume that $\Omega$ is normalized, that is 
 $$B_R(0)\subset \Omega\subset B_{nR}(0)~\text{for some}~R>0.$$
  Let $\alpha=\|u\|_{L^{\infty}(\Omega)}>0$ and $v=u/\alpha$. Then, $v\in C(\overline{\Omega})\cap C^{\infty}(\Omega)$ with
$v=0$ on $\p\Omega$, and $\|v\|_{L^{\infty}(\Omega)}=1$. Furthermore, $v$ satisfies
$\alpha^n \det D^2 v= (\alpha |v|+\e)^{-p}$ in $\Omega$.
It follows that
$$\frac{1}{(\alpha +\e)^p}\leq \alpha^n \det D^2 v\quad\text{in }\Omega.$$
Integrating both sides over $B_{R/2}(0)$, we find
$$\frac{1}{(\alpha +\e)^p} |B_{R/2}(0)| \leq \alpha^n \int_{B_{R/2}(0)} \det D^2 v~dx.$$
Now we estimate $\int_{B_{R/2}(0)}\det D^2 v~dx$ from above. Since, $B_R(0)\subset\Omega\subset B_{nR}(0)$, the convexity of $v$ and the fact that $v=0$ on $\p\Omega$ give for $x\in B_{R/2}(0)$
$$|Dv(x)|\leq \frac{|v(x)|}{\dist(x,\Omega)} \leq \frac{\|v\|_{L^{\infty}(\Omega)}}{\dist(x,\Omega)}\leq 2R^{-1}.$$
Hence
\begin{equation*}\int_{B_{R/2}(0)}\det D^2 v ~dx = |Dv(B_{R/2}(0))|\leq |B_{2R^{-1}}(0)|.
\end{equation*}
Therefore
$$\frac{1}{\alpha^n(\alpha +\e)^p} \leq  \frac{||B_{2R^{-1}}(0)||}{|B_{R/2}(0)|}= 4^n R^{-2n}\leq C(n) |\Omega|^{-2}.$$
By choosing $\e_0^n(2\e_0)^p C(n) |\Omega|^{-2}<1/2$, we obtain the conclusion of the lemma.
  \end{proof}
  Now, we construct supersolutions to (\ref{SMAp}) with optimal global H\"older regularity.
\begin{lem} [Supersolutions for (\ref{SMAp})]
\label{suplem}
Assume $p\geq 1$. 
Let $$\Omega=\{(x', x_n): |x^{'}|<1, 0<x_n < 1- |x^{'}|^2\}.$$
Then there is a constant $C= C(n, p)$ such that the function $$w= Cx_n -Cx_n^{\frac{2}{n+ p}} (1-|x'|^2)^{\frac{n+ p-2}{n+ p}}$$ is smooth, convex in $\Omega$ and satisfies
$$\det D^2 w\leq |w|^{-p}\quad\text{in } \Omega,\quad\text{and } w=0\quad\text{on }\p\Omega.$$
\end{lem}
\begin{proof} For $x=(x', x_n)$, we denote $r=|x'|$.
Let
$$v(x)= -Cx_n^a (1-r^2)^b$$
where $0<a, b<1$ and $C>0$. Then
\begin{align*}
v_{r}&= 2Cb x_n^a (1-r^2)^{b-1} r\\
v_{rr}&= 2Cb x_n^a (1-r^2)^{b-2}[1-(2b-1) r^2]\\
v_{x_n}&= -Ca x_n^{a-1}(1-r^2)^b\\
v_{x_n x_n}&= Ca(1-a) x_n^{a-2}(1-r^2)^b\\
v_{x_n r} &=2Cab x_n^{a-1} (1-r^2)^{b-1} r.
\end{align*}
In suitable coordinate systems, such as cylindrical in $x'$, the Hessian of $v$ has the following form
\begin{equation*}
 D^2 v =
 \begin{pmatrix}
  \frac{v_r}{r} & 0 & \cdots & 0 & 0 \\
  0 & \frac{v_r}{r} & \cdots & 0 & 0 \\
  \vdots  & \vdots  & \ddots & \vdots & \vdots  \\
  0 & 0 & \cdots & v_{rr} & v_{r x_n}\\
  0 & 0 & \cdots & v_{r x_n} & v_{x_n x_n}
 \end{pmatrix}.
\end{equation*}
We have
\begin{eqnarray*}\det D^2 v&=&(\frac{v_r}{r})^{n-2} [v_{x_n x_n}v_{rr}- v^2_{x_n r}]\\&=&  C^n (2b)^{n-1} x_n^{na-2}(1-r^2)^{n(b-1)} a [1-a + (1-2b-a)r^2].
\end{eqnarray*}
It follows that $v$ is convex in $\Omega$ provided that
$$0<a, b<1, \quad 1-a + (1-2b-a)r^2>0\text{ in }\Omega.$$
Since $r<1$ in $\Omega$, the last condition is equivalent to $$0\leq 1-a + (1-2b-a) = 2(1-a-b),\quad\text{or } a+ b\leq 1. $$
We would like to have $$\det D^2 v \leq |v|^{-p}= C^{-p} x_n^{-ap} (1-r^2)^{-bp}$$
which is equivalent to
\begin{equation}
\label{abC}
C^{n+p} (2b)^{n-1} x_n^{(n+ p)a-2}(1-r^2)^{(n + p)b-n } a [(1-a) (1-r^2) + 2(1-a-b) r^2] \leq  1 \quad\text{for all } r<1.
\end{equation}
Requiring $v$ to be $C^{\frac{2}{n+p}} (\overline{\Omega})$, we choose
$$a= \frac{2}{n+ p}.$$
Since $a+ b\leq 1$, (\ref{abC}) then implies that
\begin{equation}\label{bCeq} C^{n+ p} (1-r^2)^{(n+p) b-n +1} (2b)^{n-1} a(1-a)\leq 1\quad\text{for all } r<1.
\end{equation}
Note that $$(n+p) b-n +1 \leq (n+ p)(1-a) - n+ 1= p-1$$ with equality if and only $b=1-a$.

Since we want (\ref{bCeq}) to hold for $p>0$ as small as we want, a nature choice is to choose
$$b= 1-a=\frac{n+ p-2}{n+ p}.$$
Then (\ref{abC}) is exactly (\ref{bCeq}) and
(\ref{bCeq}) holds for a suitable $C=C(n,p)$ for each $p\geq 1$.

Set
$$w= Cx_n + v= C[x_n -x_n^a (1-r^2)^b].$$
Then $w$ is smooth, convex in $\Omega$ and $w=0$ on $\p\Omega$. Moreover, since
$|w|= |v|-Cx_n,$
we have
$$\det D^2 w= \det D^2 v\leq |v|^{-p}=|Cx_n + |w||^{-p}\leq |w|^{-p}\quad\text{in }\Omega.$$
\end{proof}
\begin{rem} 
\label{p1rem}
When $p=1$, the constant $C(n, 1)$ in Lemma \ref{suplem} is given by
$$C(n, 1)=(n+1)[2(n-1)]^{-\frac{n}{n+1}}.$$
The proof of Lemma  \ref{suplem} (see (\ref{abC})) shows that if $$\Omega= \{(x', x_n): |x'|<1, x_n>0\},$$ then
$$v(x) =-C(n, 1) x_n^{\frac{2}{n+1}} (1-|x'|^2)^{\frac{n-1}{n+ 1}}$$
is a solution to the singular Monge-Amp\`ere equation
$$\det D^2 v=|v|^{-1} \quad\text{in }\Omega,\quad v=0\quad\text{on }\p\Omega.$$
\end{rem}
\begin{rem}
\label{suprem2}
In the proof of Lemma \ref{suplem},
if we choose $$1-2b-a =0, \quad\text{or } b =\tilde{b}:=\frac{1-a}{2}=\frac{n+ p-2}{2(n+p)},$$ then from $$1-a+ (1-2b-a)r^2 =(1-a) (1-r^2) + 2(1-a-b) r^2 = 1-a,$$ we find that (\ref{abC}) is equivalent to
\begin{equation}\label{bCeq2} C^{n+ p} (1-r^2)^{(n+p) b-n} (2b)^{n-1} a(1-a)\leq 1\quad\text{for all } r<1,
\end{equation}
 which requires
$$0\leq (n+ p) b- n =\frac{p-n-2}{2},\quad \text{or } p\geq n+ 2.$$
Thus, for $p\geq n+2$ and a suitable $C=C(n, p)$, the function $$\tilde v(x', x_n)=-Cx_n^{\frac{2}{n+ p}} (1-|x'|^2)^{\frac{n+ p-2}{2(n+ p)}} $$ satisfies $$\det D^2 \tilde v\leq |\tilde v|^{-p}\quad\text{in } \{(x', x_n): x_n>0, |x'|<1\}.$$
When $p= n+ 2$ and $n=2$, we recover the function (\ref{uJL}) obtained in \cite{JL_JDE18}; see also \cite{JL_SCM18} for the case $n\geq 2$.

Modifying the last step in the proof of Lemma \ref{suplem}, we find that if $x_n^2 = 1-|x'|^2$ on $\p\Omega\setminus\{x_n=0\}$, then $$\tilde v(x', x_n)= - Cx_n^{\frac{2}{n+ p}} (1-|x'|^2)^{\frac{n+ p-2}{2(n+ p)}} = -Cx_n\quad\text{on } \p\Omega\setminus\{x_n=0\}.$$ We summarize these calculations in Lemma \ref{suplem2} below. 
\end{rem}
\begin{lem}[Supersolutions for (\ref{SMAp})]
\label{suplem2}
Assume $p\geq n+2$. 
Let $$\Omega=\{(x', x_n): |x^{'}|<1, 0<x_n < \sqrt{1- |x^{'}|^2}\}.$$
Then there is a constant $C= C(n, p)$ such that the function $$w= Cx_n -Cx_n^{\frac{2}{n+ p}} (1-|x'|^2)^{\frac{n+ p-2}{2(n+ p)}}$$ is smooth, convex in $\Omega$ and satisfies
$$\det D^2 w\leq |w|^{-p}\quad\text{in } \Omega,\quad\text{and } w=0\quad\text{on }\p\Omega.$$
\end{lem}
We state the following rescaled version of Lemma \ref{suplem}.
\begin{lem}
\label{suplemt}
Assume $p\geq 1$ and $t>0$.
Let $$\Omega_t=\{(x', x_n): |x^{'}|<t, 0<x_n < t^2- |x^{'}|^2\}.$$
Let $C=C(n,p)$ be as in Lemma \ref{suplem}. Then, the function
 $$w^t(x', x_n)= Ct^{\frac{2(1-p)}{n+p}}[x_n -x_n^{\frac{2}{n+ p}} (t^2-|x'|^2)^{\frac{n+ p-2}{n+ p}}]$$ is smooth, convex in $\Omega_t$ and satisfies
$$\det D^2 w^t\leq |w^t|^{-p}\quad\text{in } \Omega_t,\quad\text{and } w^t=0\quad\text{on }\p\Omega_t.$$
\end{lem}
\begin{proof}
Let $a= \frac{2}{n+ p}, b=\frac{n+ p-2}{n+ p}$, $\Omega$ and $w$ be as in Lemma \ref{suplem}. Observe that $(x', x_n)\in\Omega_t\Leftrightarrow  \left (\frac{x'}{t}, \frac{x_n}{t^2}\right)\in \Omega.$ Note that
\begin{eqnarray*}
w^t(x', x_n)= Ct^{\frac{2(1+n)}{n+p}}\left[\frac{x_n}{t^2} -\left(\frac{x_n}{t^2}\right)^a \left[1-\frac{|x'|^2}{t^2}\right]^b\right]
= t^{\frac{2(1+n)}{n+p}} w \left (\frac{x'}{t}, \frac{x_n}{t^2}\right).
\end{eqnarray*}
A direct calculation gives
$$\det D^2 w^t(x', x_n)= t^{-\frac{2(1+n) p}{n+p}}(\det D^2 w) \left (\frac{x'}{t}, \frac{x_n}{t^2}\right).$$
 The properties of $w^t$ now follow from those of $w$.
\end{proof}
  \section{Proofs of Theorem \ref{SMApthm} and Proposition \ref{q0n2}}
\label{pfsect}
\begin{proof}[Proof of Theorem \ref{SMApthm}]
(i) We divide the proof into two steps.

{\it Step 1: Uniform global H\"older estimate for approximate solutions of (\ref{SMAp}). }
Let $\{\Omega_k\}_{k=1}^{\infty}$ be a sequence of open, bounded, smooth and uniformly convex domains in $\R^n$
such that
$\Omega_k$ converges to $\Omega$ in the Hausdorff distance. Let $\e_0$ be as in Lemma \ref{ulb0}. For each $k$, consider the Monge-Amp\`ere equation
 \begin{equation*}
\left\{
 \begin{alignedat}{2}
   \det D^2 u_k~& = (|u_k|+\e_k)^{-p}~&&\text{in} ~  \Omega_k, \\\
u_k &= 0~&&\text{on}~ \p\Omega_k
 \end{alignedat} 
  \right.
  \end{equation*}
  where 
  $\e_k= \frac{\e_0}{k}.$
  From \cite[Theorem 7.1]{CNS}, there exists a unique convex solution $u_k\in C^{\infty}(\overline{\Omega_k})$ to the above equation.
  Since $|\Omega_k|\rightarrow |\Omega|$ when $k\rightarrow\infty$, we deduce from Lemma \ref{ulb0} the existence of a constant $c(n, p)>0$ such that
  \begin{equation}
  \label{uklb}
  |u_k\|_{L^{\infty}(\Omega_k)}\geq c(n, p) |\Omega|^{\frac{2}{n+ p}}.
  \end{equation}
  Let 
  $$C_{\alpha, k}= \frac{1 +2 [\diam\Omega_k]^2}{\alpha (1-\alpha)}, \quad \text{where }\alpha:=\frac{2}{n+ p}.$$
  For each $k\geq 1$ and $z\in\Omega_k$, we claim that
  \begin{equation}
  \label{ukval}
  |u_k(z)|\leq C_{\alpha, k} [\dist(z,\p\Omega_k)]^\alpha.
  \end{equation}
  Let $z$ be an arbitrary point in $\Omega_k$. By translation and rotation of coordinates, we can assume that: the origin $0$ of $\R^n$ lies on $\p\Omega_k$,  
 the $x_n$-axis points inward $\Omega_k$, $z$ lies on the $x_n$-axis, and the minimum distance to the boundary of $\Omega_k$ from $z$ is achieved at the origin.  Consider
 $$v_\alpha=  x_n^{\alpha} (|x'|^2 -C_{\alpha, k}).$$
 By Lemma \ref{valphaeq}, $\det D^2 v_\alpha\geq |v_\alpha|^{-p}$. Thus, from  $\det D^2 u_k\leq |u_k|^{-p}$, we can apply Lemma \ref{complem} to deduce that $u_k\geq v_\alpha$ in $ \Omega_k.$ This implies (\ref{uklb}).

 {\it Step 2: Convergence of a subsequence of $\{u_k\}$ to the solution of (\ref{SMAp}).}
 From the convexity of $u_k$,
we find that the functions $u_k$ have uniformly bounded global $C^{0, \frac{2}{n+p}}$ norm on $\overline{\Omega_k}$. By the Arzela-Ascoli theorem, there 
exists a subsequence of $\{u_k\}$, still
denoted by $\{u_k\}$, that converges locally uniformly to a convex function $u$ on $\Omega$. The 
estimate (\ref{ukval}) shows that $u=0$ on $\p\Omega$. Moreover, by
(\ref{uklb}), we have
$$ |u\|_{L^{\infty}(\Omega)}\geq c(n, p) |\Omega|^{\frac{2}{n+ p}}.$$
In particular, $|u|>0$ in $\Omega$.
The stability theorem
of the Monge-Amp\`ere equation (see \cite[Proposition 2.6]{Fi} and \cite[Lemma 1.2.3]{G}) then gives that the function $u$ is actually an Aleksandrov solution of (\ref{SMAp}).
Clearly, $u\in C^{\infty}(\Omega)$; see, for example \cite[Theorem 5]{CY1} or \cite[Proposition 2.8]{LSNS}. The uniqueness of $u$ follows from the comparison principle in Lemma \ref{complem}.\\
(ii) The same argument as in (\ref{ukval}) applied to $\Omega$ instead of $\Omega_k$ gives
\begin{equation}
\label{ulb}
|u(x)|\leq C(n, p,\text{diam} (\Omega)) [\dist(x,\p\Omega)]^{\frac{2}{n+ p}}.
\end{equation}
By the convexity of $u$, we easily obtain $u\in C^{\frac{2}{n+p}}(\overline{\Omega})$.\\
(iii) Let $\Omega$ be a bounded convex domain in $\R^n$ that contains parts of hyperplanes on its boundary.   Let $u\in C^{\infty}(\Omega)\cap C(\overline{\Omega})$ be the solution to (\ref{SMAp}). We show that $u\not \in C^{\beta}(\overline{\Omega})$ for any $\beta>\frac{2}{n+ p}$. Indeed, by translating and rotating coordinates, we can assume that for some $t>0$,
$$\Omega_t:=\{(x', x_n): |x^{'}|<t, 0<x_n < t^2- |x^{'}|^2\}\subset\Omega, \quad \text{and } \{(x',0): |x'|\leq t\}\subset\p\Omega.$$
Let $w^t$ be as in Lemma \ref{suplemt}. Then, $$\det D^2 w^t\leq |w^t|^{-p} \quad\text{in }\Omega_t.$$ From the convexity of $u$ and $\Omega_t\subset\Omega$, we have $u\leq 0$ on $\p\Omega_t$. Now, the comparison principle in Lemma \ref{complem} implies that
$w^t\geq u$ in $\Omega_t$
so
$$|u(x)|\geq |w^t(x)|= Ct^{\frac{2(1-p)}{n+p}}\left[x_n^{\frac{2}{n+ p}} (t^2-r^2)^{\frac{n+ p-2}{n+ p}} - x_n\right].$$
For $x=(0, x_n)$, we have then
$$|u(0, x_n)|\geq Ct^{\frac{2(1-p)}{n+p}}[t^{\frac{2(n+ p-2)}{n+ p}} x_n^\frac{2}{n+ p}-x_n] \geq \frac{C t^{\frac{n-1}{n+p}}}{2}x_n^\frac{2}{n+ p}$$
if $x_n>0$ is small, depending only on $n, p, t$. This estimate show that the exponent $2/(n+p)$ in the upper bound for $u$ in (\ref{ulb}) is optimal.
\end{proof}

\begin{proof}[Proof of Proposition \ref{q0n2}] 
Let $K:= \|u\|_{L^{\infty}(\Omega)}.$ Then,  by \cite[Lemma 3.1 (iii)]{LSNS}, we have the following uniform estimate
$$c(n, q)|\Omega|^{\frac{2}{n-q}}\leq  K\leq C(n, q))|\Omega|^{\frac{2}{n-q}}. $$
By \cite[Proposition 2.8]{LSNS},  we have $u\in C^{\infty}(\Omega)$. 
By the convexity of $u$, the global regularity $u\in C^{0,\beta}(\overline{\Omega})$ for all $\beta\in (0, \frac{2}{n-q})$ follows from the boundary estimate (\ref{uqb}). We will prove this estimate by an iterative argument  in what follows. The proof is similar to that of \cite[Proposition 5.3]{LSNS}.

Let $z$ be an arbitrary point in $\Omega$. By translation and rotation of coordinates, we can assume that: the origin $0$ of $\R^n$ lies on $\p\Omega$,  
 the $x_n$-axis points inward $\Omega$, $z$ lies on the $x_n$-axis, and the minimum distance to the boundary of $\Omega$ from $z$ is achieved at the origin.

 By the convexity of $u$, in order to verify (\ref{uqb}) at $z$, it suffices to prove that for all $x\in\Omega$, we have
 \begin{equation}
 \label{uqb2}
 |u(x)|\leq C(n,\alpha, \diam \Omega) x_n^{\beta} ~\text{for all}~\beta\in (0,  \frac{2}{n-q}). 
 \end{equation}
 {\it Step 1: When $\beta= \frac{2}{n}$. } In this step, we show 
\begin{equation}
\label{b2n}
 |u(x)| \leq C(K, n, \diam \Omega) x_n^{\frac{2}{n}}~\text{for all}~x\in\Omega.
\end{equation}
As in Lemma \ref{vallem}, for $\alpha\in (0, 1)$, we consider
\begin{equation}v_{\alpha}(x)= x_n^{\alpha} (|x'|^2 -C_\alpha)~\text{where}~C_{\alpha}= \frac{1 +2 [\diam\Omega]^2}{\alpha (1-\alpha)}.
\label{alphaeq}
\end{equation}
Then, 
 $v_{\alpha}$ is convex in $\Omega$ with
\begin{equation}\det D^2 v_{\alpha}(x)\geq 2x_n^{n\alpha-2}~\text{in}~\Omega~ \text{and} ~v_{\alpha}\leq 0~\text{on}~\p\Omega.
\label{valn}
\end{equation}
Let 
$$v:= K^{q/n} v_{\frac{2}{n}}.$$
Then, since $\det D^2 v_{2/n}\geq 2,$ we have
$$\det D^2 v = K^q \det D^2 v_{\frac{2}{n}} \geq 2K^q > \det D^2 u.$$
Note that on $\p\Omega$, $u=0\geq v$. Since $u$ and $v$ are $C^2$ in $\Omega$, we can use a simple maximum principle argument to show that
$u\geq v$ in $\Omega.$
Therefore $|u|\leq |v|$ which shows that
$$|u(x)|=|u(x', x_n)| \leq |v(x)|\leq K^{q/n} C_{\frac{2}{n}} x_n^{\frac{2}{n}}~\text{for all}~ x\in\Omega.$$
{\it Step 2: Iterative argument:} We show that, if for some $\beta_k\in (0, \frac{2}{n-q})$ we have
\begin{equation} |u(x)|\leq C(n,q, \beta_k, \diam \Omega) x_n^{\beta_k}~\text{for all}~x\in\Omega,
\label{ubk}
\end{equation}
then for all $x\in\Omega$,
\begin{equation}|u(x)|\leq \overline C(n,q, \beta_k, \diam \Omega) x_n^{\beta_{k+1}}~\text{where }~
\beta_{k+1}:= \frac{\beta_k q+ 2}{n}.
\label{ubk1}
\end{equation}
Note that if $\beta_k<\frac{2}{n-q}$ then
 $\beta_{k+1}<\frac{2}{n-q}$ 
 and
 \begin{equation}
 \label{diffbk}
 \frac{2}{n-q}-\beta_{k+1}= (\frac{2}{n-q}-\beta_k)\frac{q}{n}.
 \end{equation}
Suppose we have (\ref{ubk}). Then
for $\hat C= \hat C(n, q, \beta_k,\diam \Omega)$ large, we have 
\begin{equation}|u(x)|^{\frac{q}{n}}\leq [C(n,q, \beta_k,\diam\Omega)]^{\frac{q}{n}} x_n^{\frac{q\beta_k}{n}} < \hat C x_n^{\frac{n\beta_{k+1}-2}{n}}~\text{in}~\Omega.
\label{Cq}
\end{equation}

Denote by $(U^{ij})=(\det D^2 u) (D^2 u)^{-1}$ the cofactor matrix of the Hessian matrix $D^2 u = (D_{ij}u)$. Then $$\det U=(\det D^2 u)^{n-1} ~\text{and}~U^{ij} D_{ij}u=n \det D^2 u= n |u|^q.$$ Using (\ref{valn}), (\ref{Cq}) and the matrix inequality
$$\trace (AB)\geq n (\det A)^{1/n} (\det B)^{1/n}~\text{for~} A, B~\text{symmetric}~\geq 0,$$ we find that 
\begin{eqnarray}U^{ij} D_{ij}(\hat C v_{\beta_{k+1}}) \geq n\hat C (\det D^2 u)^{\frac{n-1}{n}} (\det D^2 v_{\beta_{k+1}})^{\frac{1}{n}}
 &\geq& n\hat C |u|^{\frac{q(n-1)}{n}} x_n^{\frac{n\beta_{k+1}-2}{n}}\nonumber \\&>& n |u|^q = n\det D^2 u=U^{ij} D_{ij}u~\text{in}~\Omega.
 \label{incr_beta}
\end{eqnarray}
Now, the maximum principle for the operator $U^{ij}D_{ij}$ applied to $u$ and $\hat C v_{\beta_{k+1}}$ gives
$u\geq \hat C v_{\beta_{k+1}}~\text{in}~\Omega. $
It follows that $$|u(x)|=|u|(x', x_n)\leq -\hat C v_{\beta_{k+1}} (x', x_n) \leq \hat C C_{\beta_{k+1}} x_n^{\beta_{k+1}}~ \text{for all}~ x\in\Omega.$$
This gives (\ref{ubk1}). 

{\it Step 3: Conclusion}. 
From {\it Step 1}, we can  choose $\beta_0= \frac{2}{n}$ to initiate {\it Step 2} and obtain a sequence $\beta_k$. From (\ref{diffbk}), we find
$$\frac{2}{n-q}-\beta_k=(\frac{2}{n-q}-\beta_0) \left(\frac{q}{n}\right)^k=\frac{2q}{n(n-q)} \left(\frac{q}{n}\right)^k.$$
Given $\beta\in (0, \frac{2}{n-q})$, we can find a positive integer $k$ such that 
$$\frac{2q}{n(n-q)} \left(\frac{q}{n}\right)^k< \frac{2}{n-q}-\beta.$$
With this $k$, we have $\beta<\beta_k<\frac{2}{n-q}$. The proposition follows by applying {\it Step 2} $k$ times.
\end{proof}

\section{Proofs of Theorem \ref{SMAkthm} and Proposition \ref{mixc}}
\label{ksect}
The outline of the proof of Theorem \ref{SMAkthm} is similar to that of Theorem \ref{SMApthm}.  For the global H\"older regularity, we use the following construction of subsolutions which is similar to Lemma \ref{vallem}. The main difference here is to take into account the origin being in the interior of the convex domain to improve the H\"older exponent.
\begin{lem}[Subsolutions for (\ref{SMAnk22})]
\label{vallemk}
Let $\Omega$ be a bounded convex domain such that $x_0=({(x_0})',-\gamma)\in \p\Omega$ and $\Omega\subset \{x=(x', x_n)\in\R^n: x_n>-\gamma\}$ where $\gamma\geq\gamma_0>0$. Let $k\geq 0$. 
Then for $a= \frac{2+ k}{2n+ 2k+ 2}\in (0, 1)$ and $C= C(n, k,\gamma_0,\diam (\Omega))$ large, the following function 
\begin{equation}
\label{valphaeqk}
v_{a}(x)= (x_n+ \gamma)^{a} (|x'|^2 -C).
\end{equation}
is smooth, convex in $\Omega$ and
satisfies
$$(\det D^2 v_a) (x\cdot Dv_a(x)- v_a(x))^k |v_a|^{n+ k+ 2}\geq 1\quad\text{on }\Omega,\quad \text{and } v_a\leq 0 \quad\text{on }\p \Omega.$$
\end{lem}
\begin{proof} For $a\in (0, 1)$ to be chosen, let $v_a(x)=  (x_n+ \gamma)^{a} (r^2 -C)$ where $r=|x'|$ and $C>[\diam (\Omega)]^2$. We calculate as in Lemmas \ref{vallem} and \ref{suplem} that $v_a$ is convex and 
\begin{eqnarray*}
\det D^2 v_a &=& (\frac{v_{a, r}}{r})^{n-2} [v_{a, x_n x_n}v_{a, rr}- v^2_{a, x_n r}]\\
&=& 2^{n-1}  (x_n+\gamma)^{an-2} [(a-a^2) C- a(1+ a)r^2].
\end{eqnarray*}
Moreover, since $\gamma\geq \gamma_0>0$,
\begin{eqnarray}
\label{vaDva}
x\cdot Dv_a - v_a &=& x_n v_{a, x_n} + r v_{a, r} - v\nonumber\\&=& (x_n +\gamma)^{a-1} [ax_n (r^2-C)+ 2r^2 (x_n+\gamma) + (C-r^2) (x_n+\gamma) ]\nonumber\\
&=&  (x_n +\gamma)^{a-1} \left\{(x_n+\gamma) [(1+ a) r^2 + (1-a) C] + a\gamma(C-r^2)\right\}\nonumber\\
&\geq & (x_n +\gamma)^{a-1} a\gamma_0 (C-r^2).
\end{eqnarray}
Therefore
\begin{multline*}
(\det D^2 v_a) (x\cdot Dv_a(x)- v_a(x))^k |v_a|^{n+ k+ 2}\\ \geq 2^{n-1}(a\gamma_0)^{k}(x_n+\gamma)^{a(2n+ 2k+ 2)- (2+ k)} [(a-a^2) C- a(1+ a)r^2] [C-r^2]^{n+ 2k+ 2}
\geq 1
\end{multline*}
if  $a= \frac{2+ k}{2n+ 2k+ 2}\in (0, 1)$ and $C= C(n, k,\gamma_0,\diam (\Omega))$ is large.
\end{proof}
For the optimality of the global H\"older exponent of solution to (\ref{SMAnk22}), we use the following construction of supersolutions which is similar to Lemma \ref{suplem}.
\begin{lem} [Supersolutions for (\ref{SMAnk22})]
\label{suplemk}
Assume $k\geq 0$ and $0<\gamma<1$. 
Let $$\Omega=\{(x', x_n): |x^{'}|<1, 0<x_n+\gamma < 1- |x^{'}|^2\}.$$
Then there is a positive constant $C_0(n, k,\gamma)$ such that the function $$w= C_0(x_n+\gamma) -C_0(x_n+\gamma)^{\frac{2+k}{2n+ 2k+2}} (1-|x'|^2)^{\frac{2n+ k}{2n+ 2k+2}}$$ is smooth, convex in $\Omega$ and satisfies
$$\det D^2 w\leq |w|^{-n-2-k} (x\cdot Dw-w)^{-k}\quad\text{in } \Omega,\quad\text{and } w=0\quad\text{on }\p\Omega.$$
\end{lem}
\begin{proof} For $0<a, b<1$ with $a+ b\leq 1$, and $C>0$ to be chosen,
let $$v= -C(x_n+\gamma)^a (1-r^2)^b\quad \text{and }w= C(x_n + \gamma) + v,\quad \text{where } r=|x'|.$$ 
As in the proof of Lemma \ref{suplem}, we know that under these conditions on $a$ and $b$, $v$ and $w$ are convex in $\Omega$. Moreover, if $a+ b=1$ then $w=0$ on $\p\Omega$. 

By the convexity of $w$, we have
$$x\cdot Dw(x)- w(x)\geq 0\cdot Dw(x)- w(0)=C\gamma^a-C\gamma>0$$
and
$$x\cdot Dw(x)- w(x)= x\cdot Dv(x)- v(x) -C\gamma\leq x\cdot Dv(x)- v(x).$$
Since $\det D^2 w=\det D^2 v$ and $|w|\leq |v|$ in $\Omega$, in order to obtain the desired properties of $w$, it suffices to prove that for $a=\frac{2+k}{2n+ 2k+2}$, $b= 1-a$, we have for suitable $C$
\begin{equation}
\label{abCk}
\det D^2 v\leq |v|^{-n-2-k} (x\cdot Dv-v)^{-k}\quad\text{in } \Omega.
\end{equation}
We compute
\begin{align*}
\det D^2 v&= C^n (2b)^{n-1} (x_n+\gamma)^{na-2}(1-r^2)^{n(b-1)} a [1-a + (1-2b-a)r^2]\\
&=C^n (2b)^{n-1} (x_n+\gamma)^{na-2}(1-r^2)^{n(b-1)+1} a(1-a),\\
x\cdot Dv- v &= C (x_n+\gamma)^{a-1} (1-r^2)^{b-1} \left[ (x_n+\gamma- ax_n) (1-r^2) + 2br^2 (x_n+\gamma)\right]\\
&\leq 3C (x_n+\gamma)^{a-1} (1-r^2)^{b-1}.
\end{align*}
It follows that
\begin{multline*}(\det D^2 v) |v|^{n+ k+ 2} (x\cdot Dv-v)^k\\
\leq 3^k (2b)^{n-1} a(1-a)C^{2n+ 2k+ 2} (x_n +\gamma)^{na-2 + a(n+ k+ 2) + (a-1) k} (1-r^2)^{n(b-1) + 1 + b(n+ k+ 2) + (b-1)k}\\
=3^k (2b)^{n-1} a(1-a)C^{2n+ 2k+ 2} (x_n +\gamma)^{a(2n+ 2k+ 2)-2-k} (1-r^2)^{b(2n+ 2k+ 2)-n-k + 1}.
\end{multline*}
Thus (\ref{abCk}) holds for a suitable $C=C_0(n, k,\gamma)$ when $a=\frac{2+k}{2n+ 2k+2}$ and $b= 1-a$.
\end{proof}

\begin{proof}[Proof of Theorem \ref{SMAkthm}]  First, we prove (\ref{uHk}) for the unique convex solution $u\in C^{\infty}(\Omega)\cap C(\overline{\Omega})$ to (\ref{SMAnk22}). Let $$a=\frac{2+k}{2n+ 2k+2} \quad\text{and }\gamma_0=\dist(0,\p\Omega)>0.$$ 
Let $z$ be an arbitrary point in $\Omega$. Let $x_0\in\p\Omega$ be such that $|z-x_0|=\dist(z,\p\Omega)$. Suppose that the supporting hyperplane $l_{x_0}$ to $\p\Omega$ at $x_0$ has equation
$\vv{n}\cdot (x-x_0)=0$ and $$\Omega\subset \{x\in\R^n: \vv{n}\cdot (x-x_0)\geq 0\}.$$ Then $\gamma:= -\vv{n}\cdot x_0= \dist (0, l_{x_0})\geq \gamma_0$. From Lemma \ref{vallemk}, we find that
for a suitable $C= C(n, k, \diam (\Omega),\gamma_0)$, the function $$v(x):= [\vv{n}\cdot (x-x_0)]^a \left(|x-\vv{n}\cdot x|^2 -C\right)$$ is a subsolution to (\ref{SMAnk22}).  To see this, we can use a rotation to assume that $\vv{n}= (0,\cdots, 0, 1)$ and hence $x_0=((x_0)^{'}, -\gamma)$.

Using the comparison principle in Lemma \ref{complem}, we find $u\geq v$ and hence
\begin{equation}
\label{upua}
|u(z)|\leq |v(z)|\leq C |z-x_0|^a = C[\dist(z,\p\Omega)]^a. 
\end{equation}
This holds for all $z\in\Omega$ so (\ref{uHk}) is proved. By the convexity of $u$, we easily obtain $u\in C^a(\overline{\Omega})$.

Finally, we note that the optimality of the exponent $a$ follows from Lemma \ref{suplemk} and the comparison principle. Indeed, let $\Omega$, $w$ and $C_0$ be as in Lemma \ref{suplemk}. Let $u\in C^{\infty}(\Omega)\cap C(\overline{\Omega})$ be the unique convex solution to (\ref{SMAnk22}). Since $w$ is a supersolution to (\ref{SMAnk22}), by  the comparison principle in Lemma \ref{complem}, we find that $w\geq u$ in $\Omega$. Hence for $x=(0, x_n)\in\Omega$ where $-\gamma<x_n<1-\gamma$, we have
\begin{equation}
\label{lowua}
|u(x)|\geq |w(x)| = C_0[(x_n +\gamma)^a- (x_n+\gamma)]\geq C_0(x_n +\gamma)^a/2= C_0[\dist(x,\p\Omega)]^a/2
\end{equation}
if $x_n+\gamma>0$ is small. Hence,  the exponent $a$ in $u\in C^a(\overline{\Omega})$ is optimal.
\end{proof}
\begin{proof}[Proof of Proposition \ref{mixc}] 
Let $u\in C^{\infty}(\Omega)\cap C(\overline{\Omega})$ be the unique convex solution to (\ref{SMAnk22}) where $$\Omega=\{(x', x_n): |x^{'}|<1, 0<x_n+\gamma < 1- |x^{'}|^2\}.$$
Let $C= C(n, k,\gamma,\diam (\Omega))$ and $C_0= C_0(n,k,\gamma)$ be as in Lemmas \ref{vallemk} and \ref{suplemk}, respectively. Let $$a= \frac{2+k}{2n+ 2k+2}.$$ 
The main technical point of the proof is to obtain a positive lower bound comparable to $[\dist(x,\p\Omega)]^{a-1}$ for $x\cdot Du - u$; see (\ref{xDuu}).

Consider $x= (0, x_n)\in\Omega$ with $0<x_n+\gamma$ is small. Thus $x_n<0$. 
As in (\ref{upua}), we have
\begin{equation}
\label{upua2}
-u(x)=-u(0, x_n) =|u(0, x_n)|\leq C(x_n +\gamma)^a.
\end{equation}
For $m>0$ large to be chosen, we find from (\ref{lowua}) that as long as $(0, x_n + m(x_n+\gamma))\in\Omega$
$$u(0, x_n + m(x_n +\gamma)) \leq C_0\left\{(m+1)(x_n +\gamma)-[(m+1)(x_n +\gamma)]^a\right\}.$$
Hence, by the convexity of $u$, we have 
\begin{eqnarray*}u_{x_n}(0, x_n)&\leq& \frac{u(0, x_n + m(x_n +\gamma))- u(0, x_n)}{m(x_n+\gamma)}\\
&\leq& C_0\frac{m+ 1}{m} - (x_n +\gamma)^{a-1} \frac{C_0 (m+ 1)^a -C}{m}.
\end{eqnarray*}
We first choose $m$ large such that $$C_0 (m+ 1)^a -C> C_0/2.$$
Then
$$u_{x_n}(0, x_n) <C_0\frac{m+ 1}{m} -(x_n +\gamma)^{a-1} \frac{1}{2m}. $$
Now, recalling (\ref{lowua}), we choose $0<x_n +\gamma\leq \bar\gamma$ small so that $-\gamma<x_n<-\gamma/2$ and 
\begin{equation}
\label{lowua2}
|u(x)\geq C_1 (x_n +\gamma)^a,\quad u_{x_n}(0, x_n) < -C_1(x_n +\gamma)^{a-1},\quad C_1:= \min\{C_0/2, \frac{1}{4m}\}. 
\end{equation}
With these choices of $m$ and $x_n$, we have for $x=(0, x_n)$
\begin{equation}
\label{xDuu}
x\cdot Du - u>x_n u_{x_n}(0, x_n) >-C_1 x_n (x_n +\gamma)^{a-1}> C_2 (x_n +\gamma)^{a-1},\quad C_2=\frac{C_1\gamma}{2}. 
\end{equation}
Using this estimate together with (\ref{lowua2}) and recalling the definition of $a$, we find 
\begin{eqnarray*}f(x)&=& |u(x)|^{-n-2-k} (x\cdot Du(x) -u(x))^{-k}\\& \leq& C^{-n-2-k} C_2^{-k}  (x_n +\gamma)^{-a(n+2+ k)-k(a-1)} 
\\&=& C_3 (x_n +\gamma)^{\frac{(n-4)k-(2n+ 4)}{2n+ 2k+ 2}} \\&=& C_3 [\dist(x,\p\Omega]^{\frac{(n-4)k-(2n+ 4)}{2n+ 2k+ 2}}.
\end{eqnarray*}
Therefore, we obtain (\ref{fmix}), completing the proof of the proposition.
\end{proof}

\end{document}